\pdfoutput=1
\documentclass[11pt, a4paper, twoside,leqno]{amsart}
\usepackage[centering, totalwidth = 380pt, totalheight = 600pt]{geometry}
\usepackage[bbgreekl]{mathbbol}
\usepackage{amssymb, amsmath, amsthm}
\usepackage{microtype, stmaryrd, url, lmodern, eucal}
\usepackage[shortlabels]{enumitem}
\usepackage[latin1]{inputenc}
\usepackage{color}
\definecolor{darkgreen}{rgb}{0,0.45,0}
\usepackage[colorlinks,citecolor=darkgreen,linkcolor=darkgreen]{hyperref}
\usepackage[british]{babel}

\makeatletter
\def\@cite#1#2{[{#1\if@tempswa ,~#2\fi}]}
\makeatother

\DeclareMathAlphabet{\mathbf}{OT1}{cmr}{b}{n}

\usepackage[arrow, matrix, tips, curve, graph, rotate]{xy}
\SelectTips{cm}{10}

\makeatletter
\def\matrixobject@{%
  \edef \next@{={\DirectionfromtheDirection@ }}%
  \expandafter \toks@ \next@ \plainxy@
  \let\xy@@ix@=\xyq@@toksix@
  \xyFN@ \OBJECT@}
\let\xy@entry@@norm=\entry@@norm
\def\entry@@norm@patched{%
  \let\object@=\matrixobject@
  \xy@entry@@norm }
\AtBeginDocument{\let\entry@@norm\entry@@norm@patched}
\makeatother

\newcommand{\twocong}[2][0.5]{\ar@{}[#2] \save ?(#1)*{\cong}\restore}
\newcommand{\twoeq}[2][0.5]{\ar@{}[#2] \save ?(#1)*{=}\restore}
\newcommand{\rtwocell}[3][0.5]{\ar@{}[#2] \ar@{=>}?(#1)+/l 0.2cm/;?(#1)+/r 0.2cm/^{#3}}
\newcommand{\ltwocell}[3][0.5]{\ar@{}[#2] \ar@{=>}?(#1)+/r 0.2cm/;?(#1)+/l 0.2cm/^{#3}}
\newcommand{\ltwocello}[3][0.5]{\ar@{}[#2] \ar@{=>}?(#1)+/r 0.2cm/;?(#1)+/l 0.2cm/_{#3}}
\newcommand{\dtwocell}[3][0.5]{\ar@{}[#2] \ar@{=>}?(#1)+/u  0.2cm/;?(#1)+/d 0.2cm/^{#3}}
\newcommand{\dltwocell}[3][0.5]{\ar@{}[#2] \ar@{=>}?(#1)+/ur  0.2cm/;?(#1)+/dl 0.2cm/^{#3}}
\newcommand{\drtwocell}[3][0.5]{\ar@{}[#2] \ar@{=>}?(#1)+/ul  0.2cm/;?(#1)+/dr 0.2cm/^{#3}}
\newcommand{\dthreecell}[3][0.5]{\ar@{}[#2] \ar@3{->}?(#1)+/u  0.2cm/;?(#1)+/d 0.2cm/^{#3}}
\newcommand{\utwocell}[3][0.5]{\ar@{}[#2] \ar@{=>}?(#1)+/d 0.2cm/;?(#1)+/u 0.2cm/_{#3}}
\newcommand{\dtwocelltarg}[3][0.5]{\ar@{}#2 \ar@{=>}?(#1)+/u  0.2cm/;?(#1)+/d 0.2cm/^{#3}}
\newcommand{\utwocelltarg}[3][0.5]{\ar@{}#2 \ar@{=>}?(#1)+/d  0.2cm/;?(#1)+/u 0.2cm/_{#3}}

\newdir{(}{{}*!<0em,-.14em>-\cir<.14em>{l^r}}
\newdir{ (}{{}*!/-5pt/\dir{(}}
\newdir{ >}{{}*!/-5pt/\dir{>}}
\newcommand{\sh}[2]{**{!/#1 -#2/}}

\DeclareMathOperator{\adh}{adh}

\newcommand{\cat}[1]{\mathrm{\mathcal #1}}
\newcommand{\thg}{{\mathord{\text{--}}}}

\newcommand{\cd}[2][]{\vcenter{\hbox{\xymatrix#1{#2}}}}

\renewcommand{\phi}{\varphi}
\newcommand{\A}{{\mathcal A}}

\newcommand{\C}{{\mathcal C}}

\newcommand{\F}{{\mathcal F}}
\newcommand{\G}{{\mathcal G}}
\renewcommand{\H}{{\mathcal H}}

\newcommand{\M}{{\mathcal M}}

\renewcommand{\P}{{\mathcal P}}

\let\sec=\S
\renewcommand{\S}{{\mathcal S}}

\newcommand{\xtor}[1]{\cdl[@1]{{} \ar[r]|-{\object@{|}}^{#1} & {}}}
\newcommand{\tor}{\ensuremath{\relbar\joinrel\mapstochar\joinrel\rightarrow}}

\makeatletter

\def\hookleftarrowfill@{\arrowfill@\leftarrow\relbar{\relbar\joinrel\rhook}}
\def\twoheadleftarrowfill@{\arrowfill@\twoheadleftarrow\relbar\relbar}
\def\leftbararrowfill@{\arrowdoublefill@{\leftarrow\mkern-5mu}\relbar\mapstochar\relbar\relbar}
\def\Leftbararrowfill@{\arrowdoublefill@{\Leftarrow\mkern-2mu}\Relbar\Mapstochar\Relbar\Relbar}
\def\leftringarrowfill@{\arrowdoublefill@{\leftarrow\mkern-3mu}\relbar{\mkern-3mu\circ\mkern-2mu}\relbar\relbar}
\def\lefttriarrowfill@{\arrowfill@{\mathrel\triangleleft\mkern0.5mu\joinrel\relbar}\relbar\relbar}
\def\Lefttriarrowfill@{\arrowfill@{\mathrel\triangleleft\mkern1mu\joinrel\Relbar}\Relbar\Relbar}

\def\hookrightarrowfill@{\arrowfill@{\lhook\joinrel\relbar}\relbar\rightarrow}
\def\twoheadrightarrowfill@{\arrowfill@\relbar\relbar\twoheadrightarrow}
\def\rightbararrowfill@{\arrowdoublefill@{\relbar\mkern-0.5mu}\relbar\mapstochar\relbar\rightarrow}
\def\Rightbararrowfill@{\arrowdoublefill@{\Relbar\mkern-2mu}\Relbar\Mapstochar\Relbar\Rightarrow}
\def\rightringarrowfill@{\arrowdoublefill@\relbar\relbar{\mkern-2mu\circ\mkern-3mu}\relbar{\mkern-3mu\rightarrow}}
\def\righttriarrowfill@{\arrowfill@\relbar\relbar{\relbar\joinrel\mkern0.5mu\mathrel\triangleright}}
\def\Righttriarrowfill@{\arrowfill@\Relbar\Relbar{\Relbar\joinrel\mkern1mu\mathrel\triangleright}}

\def\leftrightarrowfill@{\arrowfill@\leftarrow\relbar\rightarrow}
\def\mapstofill@{\arrowfill@{\mapstochar\relbar}\relbar\rightarrow}

\renewcommand*\xleftarrow[2][]{\ext@arrow 20{20}0\leftarrowfill@{#1}{#2}}
\providecommand*\xLeftarrow[2][]{\ext@arrow 60{22}0{\Leftarrowfill@}{#1}{#2}}
\providecommand*\xhookleftarrow[2][]{\ext@arrow 10{20}0\hookleftarrowfill@{#1}{#2}}
\providecommand*\xtwoheadleftarrow[2][]{\ext@arrow 60{20}0\twoheadleftarrowfill@{#1}{#2}}
\providecommand*\xleftbararrow[2][]{\ext@arrow 10{22}0\leftbararrowfill@{#1}{#2}}
\providecommand*\xLeftbararrow[2][]{\ext@arrow 50{24}0\Leftbararrowfill@{#1}{#2}}
\providecommand*\xleftringarrow[2][]{\ext@arrow 10{26}0\leftringarrowfill@{#1}{#2}}
\providecommand*\xlefttriarrow[2][]{\ext@arrow 80{24}0\lefttriarrowfill@{#1}{#2}}
\providecommand*\xLefttriarrow[2][]{\ext@arrow 80{24}0\Lefttriarrowfill@{#1}{#2}}

\renewcommand*\xrightarrow[2][]{\ext@arrow 01{20}0\rightarrowfill@{#1}{#2}}
\providecommand*\xRightarrow[2][]{\ext@arrow 04{22}0{\Rightarrowfill@}{#1}{#2}}
\providecommand*\xhookrightarrow[2][]{\ext@arrow 00{20}0\hookrightarrowfill@{#1}{#2}}
\providecommand*\xtwoheadrightarrow[2][]{\ext@arrow 03{20}0\twoheadrightarrowfill@{#1}{#2}}
\providecommand*\xrightbararrow[2][]{\ext@arrow 01{22}0\rightbararrowfill@{#1}{#2}}
\providecommand*\xRightbararrow[2][]{\ext@arrow 04{24}0\Rightbararrowfill@{#1}{#2}}
\providecommand*\xrightringarrow[2][]{\ext@arrow 01{26}0\rightringarrowfill@{#1}{#2}}
\providecommand*\xrighttriarrow[2][]{\ext@arrow 07{24}0\righttriarrowfill@{#1}{#2}}
\providecommand*\xRighttriarrow[2][]{\ext@arrow 07{24}0\Righttriarrowfill@{#1}{#2}}

\providecommand*\xmapsto[2][]{\ext@arrow 01{20}0\mapstofill@{#1}{#2}}
\providecommand*\xleftrightarrow[2][]{\ext@arrow 10{22}0\leftrightarrowfill@{#1}{#2}}
\providecommand*\xLeftrightarrow[2][]{\ext@arrow 10{27}0{\Leftrightarrowfill@}{#1}{#2}}

\makeatother

\numberwithin{equation}{section}

\theoremstyle{plain}
\newtheorem{Thm}{Theorem}
\newtheorem{Prop}[Thm]{Proposition}
\newtheorem{Cor}[Thm]{Corollary}
\newtheorem{Lemma}[Thm]{Lemma}
\newtheorem*{Thm*}{Theorem}

\theoremstyle{definition}
\newtheorem{Defn}[Thm]{Definition}

\newcommand{\mnd}[1]{\mathbb{#1}}

\newcommand{\flt}{\mnd{F}}
\newcommand{\uf}{\bbbeta}
\newcommand{\ps}{\mnd{P}}
\newcommand{\vi}{\mnd{V}}
\newcommand{\ms}{\omega}
\newcommand{\us}{\nu}
\newcommand{\mt}{\mu}
\newcommand{\ut}{\eta}

\begin{document}
\leftmargini=2em 
\title[The Vietoris monad and weak distributive laws]{The Vietoris monad and\\weak distributive laws}
\author{Richard Garner} 
\address{Centre of Australian Category Theory, Macquarie University, NSW 2109, Australia} 
\email{richard.garner@mq.edu.au}

\date{\today}

\thanks{The support of Australian Research Council grants
  DP160101519 and FT160100393 is gratefully acknowledged.}

\begin{abstract}
  The Vietoris monad on the category of compact Hausdorff spaces is a
  topological analogue of the power-set monad on the category of sets.
  Exploiting Manes' characterisation of the compact Hausdorff spaces
  as algebras for the ultrafilter monad on sets, we give precise form
  to the above analogy by exhibiting the Vietoris monad as induced by
  a weak distributive law, in the sense of B\"ohm, of the power-set
  monad over the ultrafilter monad.
\end{abstract}
\maketitle
\newcommand{\KH}{\cat{K\H aus}}
\section{Introduction}
In his 1922 paper~\cite{Vietoris1922Bereiche}, Vietoris described how
the set of closed subspaces of a compact Hausdorff space $X$ can
itself be made into a compact Hausdorff space, now often referred to
as the \emph{Vietoris hyperspace} $VX$. The Vietoris construction is
important not just in topology, but also in theoretical computer
science, where its various generalisations provide different notions
of \emph{power~domain}~\cite{Smyth1983Power}, and in general algebra,
where its restriction to zero-dimensional spaces links up under Stone
duality with the theory of \emph{Boolean algebras with
  operators}~\cite{Jonsson1951Boolean}.

The assignation $X \mapsto VX$ in fact underlies a monad $\vi$ on the
category $\KH$ of compact Hausdorff spaces. This monad structure was
sketched briefly by Manes in~\cite[Exercise
I.5.23]{Manes1976Algebraic}, but received its first detailed treatment
by Wyler in~\cite{Wyler1981Continuous}; in particular, Wyler
identified the $\vi$-algebras as Scott's \emph{continuous
  lattices}~\cite{Scott1972Continuous}.

Clearly, the Vietoris monad is related to the \emph{power-set} monad
$\ps$ on the category of sets. In both cases, the monad unit and
multiplication are given by inclusion of singletons and
by set-theoretic union; and both underlying functors are
``power-object'' constructions---differing in the
distinction between \emph{closed} subspaces and \emph{arbitrary}
subsets, and in the need to topologise in the former case.

In this article, we give a new account of the Vietoris monad on $\KH$
which explains its similarities with the power-set monad on
$\cat{Set}$ by deriving it from it in a canonical way; for good
measure, this account also renders the slightly delicate topological
aspects of the Vietoris construction entirely automatic.

The starting point is Manes' result~\cite{Manes1969A-triple}
identifying compact Hausdorff spaces as the algebras for the
ultrafilter monad $\uf$ on $\cat{Set}$. In light of this, we recognise our
situation as the following one: we have a monad---namely, the
power-set monad $\ps$ on sets---which we would like to ``lift''
appropriately to the category of algebras for another monad on the
same base---namely, the ultrafilter monad $\uf$.

At this point, the categorically-minded reader will doubtless think of
Beck's theory~\cite{Beck1969Distributive} of distributive laws. For
monads $\mnd{S}$, $\mnd{T}$ on $\C$, a \emph{distributive law}
of $\mnd{S}$ over $\mnd{T}$ is a natural transformation
$\delta \colon TS \Rightarrow ST$ satisfying four axioms expressing
compatibility with the monad structures of $\mnd{S}$ and
$\mnd{T}$. As we will recall in
Section~\ref{sec:distributive-laws-1} below, distributive laws
correspond to \emph{liftings} of $\mnd{S}$ to a monad on the
category of $\mnd{T}$-algebras, and also to \emph{extensions} of
$\mnd{T}$ to a monad on the Kleisli category of
$\mnd{S}$.

In particular, 
we can ask: is there a distributive law of $\ps$ over $\uf$ for which
the associated lifting of $\ps$ to $\KH$, the category of
$\uf$-algebras, is the Vietoris monad? Unfortunately, the answer to
this question is \emph{no}, since the kind of lifting mandated by the
theory of distributive laws is too strong; if the Vietoris monad did
lift the power-set monad in this sense, then the underlying set of
$VX$ would comprise the full power-set of $X$, rather than just the
closed subsets.

However, we are clearly very close to having a lifting of $\ps$ to
$\uf$-algebras; and, in fact, we are also very close to having a
distributive law of $\ps$ over $\uf$. For indeed, such a distributive
law would be the same as an extension of $\uf$ to the Kleisli category
of $\ps$, which is the category $\cat{Rel}$ of sets and relations; and
the extension of structure from $\cat{Set}$ to $\cat{Rel}$ was
analysed in detail by Barr~\cite{Barr1970Relational}. As observed
in~\cite[\sec 2.11]{Trnkova1977Relational}, it is a direct consequence
of Barr's results that:
\begin{itemize}[itemsep=0.25\baselineskip]
\item A functor $F \colon \cat{Set} \rightarrow \cat{Set}$ has
  \emph{at most} one extension to a locally monotone functor
  $\tilde F \colon \cat{Rel} \rightarrow \cat{Rel}$, which exists just
  when $F$ is weakly cartesian;
 \item If $F,G$ are weakly cartesian, then
  $\alpha \colon F \Rightarrow G$ has \emph{at most} one extension to
  a natural transformation
  $\tilde \alpha \colon \tilde F \Rightarrow \tilde G$, existing
  just when $\alpha$ is weakly cartesian.
\end{itemize}
 (The definition of weak cartesianness is recalled in
 Section~\ref{sec:extend-struct-from} below.)
In the case of the ultrafilter monad $\uf$ on $\cat{Set}$, it is
well known that the underlying endofunctor and the monad
multiplication are weakly cartesian, and so extend; while the unit is
not, and so does not. This not-quite extension of $\uf$ to $\cat{Rel}$ turns
out to correspond to a not-quite distributive law
$\delta \colon \beta P \Rightarrow P\beta$, which is compatible with
both monad multiplications and the unit of $\ps$, but not with the unit of
$\uf$.

One perspective on this situation can be found
in~\cite{Clementino2004One-setting, Hofmann2014Lax-algebras,
  Tholen2019Lax-Distributive, Tholen2017Quantalic}. As was already
essentially observed in~\cite{Barr1970Relational}, the not-quite
extension of $\uf$ to $\cat{Rel}$ is an example of a \emph{lax} monad
extension in the sense of~\cite{Clementino2004One-setting}. It was
noted in~\cite[Exercise~1.I]{Hofmann2014Lax-algebras}, and confirmed
in~\cite{Tholen2019Lax-Distributive}, that such lax monad extensions
correspond to suitably-defined lax distributive laws, and further
explained in~\cite{Tholen2017Quantalic} that these correspond, in
turn, to suitable lax liftings. These facts are important for the area
of \emph{monoidal topology}; see, for example,~\cite{2014Monoidal}.

However, for the ends we wish to pursue here, a different point of
view is relevant. In 2009, with motivation from quantum algebra,
Street~\cite{Street2009Weak} and B\"ohm~\cite{Bohm2010The-weak}
introduced various notions of \emph{weak distributive law} of a monad
$\mnd{S}$ over a monad $\mnd{T}$, 
involving a natural transformation
$\delta \colon TS \Rightarrow ST$ satisfying Beck's original axioms
relating to the monad multiplications, but weakening in different ways
those relating to the monad units. Each of these kinds of weak
distributive law of $\mnd{S}$ over $\mnd{T}$ was shown to
correspond to a kind of ``weak lifting'' of $\mnd{S}$ to
$\mnd{T}$-algebras.

In particular, one of the kinds of weak distributive law involves
simply dropping from Beck's original notion the axiom relating to the
unit of $\mnd{T}$. Thus, the not-quite distributive law
$\delta \colon \beta P \Rightarrow P\beta$ we described above is a
weak distributive law, in this sense, of $\ps$ over $\uf$; and so
there is a corresponding weak lifting of $\ps$ to $\uf$-algebras. Our
main result identifies this weak lifting by proving:

\begin{Thm*}
  The Vietoris monad on the category of compact Hausdorff spaces is
  the weak lifting of the power-set monad 
  associated to the canonical weak distributive law
  of the power-set monad over the ultrafilter monad.
\end{Thm*}

As an application of this result, we obtain a simple new proof of
Wyler's characterisation of the $\mnd V$-algebras as the continuous
lattices; and we conclude the paper with remarks on possible
variations and generalisations of our main result.

\section{The monads}
\label{sec:monads}

\subsection{The power-set monads}
\label{sec:power-set-monads}
We begin  by recalling the various monads of interest and
their categories of algebras. Most straightforwardly, we
have:

\begin{Defn}
  \label{def:7}
  The power-set monad $\ps$ on $\cat{Set}$ has $PX$ given by the set
  of all subsets of $X$, and $Pf \colon PX \rightarrow PY$ given by
  direct image. The unit $\ut_X \colon X \rightarrow PX$ and
  multiplication $\mt_X \colon PPX \rightarrow PX$ are given by
  $\ut_X(x) = \{x\}$ and $\mt_X(\A) = \bigcup \A$.
\end{Defn}


The $\ps$-algebras can be identified as complete lattices in two
different ways, depending on whether we view the $\ps$-algebra
structure as providing the sup operation or the inf operation; the
maps of the category of $\ps$-algebras are then respectively the
sup-preserving maps and the inf-preserving maps.

\subsection{The ultrafilter monad}
\label{sec:ultrafilter-monad-1}
Recall that a \emph{filter} on a set $X$ is a non-empty subset
$\F \subseteq PX$ such that, for all $A,B \subseteq X$, we have
$A,B \in \F$ iff $A \cap B \in \F$. A filter is an \emph{ultrafilter}
if it contains exactly one of $A$ and $X \setminus A$ for each
$A \subseteq X$. 

 \begin{Defn}
  \label{def:2}
  The ultrafilter monad $\uf$ on $\cat{Set}$ has $\beta X$ given by the set of all ultrafilters on $X$, and
  $\beta f \colon \beta X \rightarrow \beta Y$ the function taking \emph{pushforward} along $f$:
  \begin{equation*}
    \F \mapsto f_!(\F) = \{ B \subseteq Y : f^{-1}(B) \in \F \} = \{ B \subseteq Y
    : f(A) \subseteq B \text{ for some } A \in \F\}\rlap{ .}
  \end{equation*}
  The unit $\ut_X \colon X \rightarrow \beta X$ and multiplication
  $\mt_X \colon \beta \beta X \rightarrow \beta X$ are defined by
  $\ut_X(x) = \{A \subseteq X : x \in A\}$ and
  $\mt_X(\mathbf{F}) = \{A \subseteq X : A^\# \in \mathbf{F}\}$, where
  for any $A \subseteq X$ we define
  $A^\# = \{ \F \in \beta X : A \in \F\}$.
\end{Defn}

The algebras for the ultrafilter monad were identified by
Manes~\cite{Manes1969A-triple} as the compact Hausdorff spaces. Recall
that, for a topological space $X$, an ultrafilter $\F \in \beta X$ is
said to \emph{converge} to $x \in
X$ 
if each neighbourhood of $x$ is in $\F$; and that, when $X$ is compact
Hausdorff, each $\F \in \beta X$ converges to a unique point
$\xi(\F)$. 
Manes showed that the function
$\xi \colon \beta X \rightarrow X$ so determined endows the compact
Hausdorff $X$ with
$\uf$-algebra structure, and that \emph{every} $\uf$-algebra arises
thus. Under this identification, the $\uf$-algebra maps are
the continuous ones.

%

\subsection{The Vietoris monad}
\label{sec:vietoris-monad}
The \emph{Vietoris hyperspace}~\cite{Vietoris1922Bereiche} $VX$ of a
compact Hausdorff space $X$ is the set of all closed subspaces of $X$,
endowed with the topology (sometimes called the ``hit-and-miss''
topology) generated by the following subbasic open sets for each
$C \in VX$:
\begin{equation*}
  C^+ = \{A \in VX : A \cap C = \emptyset\} \qquad \text{and} \qquad C^- =
  \{A \in VX : A \nsubseteq C\}\rlap{ .}
\end{equation*}

\begin{Defn}
  \label{def:7}\cite{Wyler1981Continuous}
  The Vietoris monad $\vi$ on $\KH$ has $VX$ given as above, and
  $Vf \colon VX \rightarrow VY$ given by direct image. The unit
  $\ut_X \colon X \rightarrow VX$ and multiplication
  $\mt_X \colon VVX \rightarrow VX$ are given by $\ut_X(x) = \{x\}$
  and $\mt_X(\A) = \bigcup \A$.
\end{Defn}

It was shown in~\cite{Wyler1981Continuous} that the $\vi$-algebras are
the \emph{continuous lattices} of~\cite{Scott1972Continuous}. Recall
that, for elements $x,y$ of a poset $L$, we write $x \ll y$ if,
whenever $D \subseteq L$ is directed and
$y \leqslant \sup D$, there exists some $d \in D$ with
$x \leqslant d$. A \emph{continuous lattice} is a complete lattice $L$
such that every $x\in L$ satisfies $x = \sup\{ s : s \ll x \}$. 
Under Wyler's identification, a continuous lattice $L$ becomes a compact
Hausdorff space under its \emph{Lawson topology}, which is generated
by the subbasic open sets
\begin{equation*}
  s^+ = \{x \in L : s \ll x\} \qquad \text{and} \qquad s^- = \{x \in L
  : s \nleqslant x\} \qquad \text{for $s \in L$}\rlap{ ,}
\end{equation*}
and a
%
%
%
$\vi$-algebra
via the function $VL \rightarrow L$ taking infima of closed sets.

\section{Distributive laws and weak distributive laws}
\label{sec:distr-laws-weak}
%
%

\subsection{Distributive laws}
\label{sec:distributive-laws-1}
We now recall Beck's classical theory~\cite{Beck1969Distributive} of
\emph{distributive laws} and their associated liftings and extensions,
and the generalisation of this theory to \emph{weak distributive
  laws}~\cite{Bohm2010The-weak} which will be necessary for our main
result. We begin with Beck's original notion.
\begin{Defn}
  \label{def:4}
  Let $\mnd{S} = (S, \us, \ms)$ and $\mnd{T} = (T, \ut, \mt)$ be
  monads on a category $\C$. A \emph{distributive law of
    $\mnd{S}$ over $\mnd{T}$} is a natural transformation
  $\delta \colon TS \Rightarrow ST$ rendering commutative the four
  diagrams:
  \begin{equation*}
    \cd[@-0.5em@C-1em]{
      \sh{l}{0.2em}TSS \ar[r]^-{\delta S} \ar[d]_-{T\ms} &
      STS \ar[r]^-{S\delta} & \sh{r}{0.2em} SST \ar[d]^-{\ms T} & 
      \sh{l}{0.2em}TTS \ar[r]^-{T\delta} \ar[d]_-{\mt S} & 
      TST \ar[r]^-{\delta T} & \sh{r}{0.2em}STT \ar[d]^-{S\mt} & 
      & {T} \ar[dl]_-{T\us} \ar[dr]^-{\us T} & &
      & {S} \ar[dl]_-{\ut S} \ar[dr]^-{S\ut}
      \\
      TS \ar[rr]^-{\delta} & & ST & 
      {TS} \ar[rr]^-{\delta} & & {ST} & 
      \sh{r}{0.2em}{TS} \ar[rr]^-{\delta} & & 
      \sh{l}{0.2em}{ST} & 
      \sh{r}{0.2em}{TS} \ar[rr]^-{\delta} & & 
      \sh{l}{0.2em}{ST}\rlap{.}
    }
  \end{equation*}
\end{Defn}


The basic result about distributive laws is that they correspond both
to liftings and to extensions, in the sense of the following definition.
\begin{Defn}
  \label{def:3}
  Let $\mnd{S} = (S, \us, \ms)$ and $\mnd{T} = (T, \ut, \mt)$ be
  monads on a category $\C$. If we write
  $U^\mnd{T} \colon \C^\mnd{T} \rightarrow \C$ for the forgetful
  functor from the category of $\mnd T$-algebras, then a
  \emph{lifting} of $\mnd S$ to $\C^\mnd{T}$ is a monad
  $\tilde{\mnd S}$ on $\C^\mnd{T}$ such that
  \begin{equation*}
    U^\mnd{T} \circ \tilde S = S \circ U^\mnd{T} \qquad
    U^\mnd{T} \circ \tilde \us = \us \circ U^\mnd{T} \quad \text{and} \quad
    U^\mnd{T} \circ \tilde \ms = \ms \circ U^\mnd{T}\rlap{ .}
  \end{equation*}
  On the other hand, if we write
  $F_{\mnd S} \colon \C \rightarrow \C_{\mnd S}$ for the free functor
  into the Kleisli category of $\mnd S$, then an \emph{extension} of
  $\mnd{T}$ to $\C_\mnd{S}$ is a
  monad $\tilde{\mnd T}$ on $\C_{\mnd S}$ such that
  \begin{equation*}
    \tilde T \circ F_{\mnd S} = F_{\mnd S} \circ T \qquad
    \tilde \ut \circ F_{\mnd S} = F_{\mnd S} \circ \ut \quad \text{and} \quad
    \tilde \mt \circ F_{\mnd S} = F_{\mnd S} \circ \mt\rlap{ .}
  \end{equation*}
\end{Defn}

\begin{Prop}
  \label{prop:2}
  \cite[\sec 1]{Beck1969Distributive},
  \cite[Theorem~2.5]{Meyer1975Induced}. 
  For monads $\mnd{S}$, $\mnd{T}$ on $\C$, there are bijections
  between distributive laws of $\mnd{S}$ over
  $\mnd{T}$, liftings of $\mnd{S}$ to $\C^{\mnd T}$ and extensions of
  $\mnd T$ to $\C_{\mnd S}$.
\end{Prop}
\begin{proof}
  Given a distributive law $\delta \colon TS \Rightarrow ST$, we
  define the corresponding lifting of $\mnd{S}$ to
  $\mnd{T}$-algebras to have action on objects given by
  \begin{equation*}
    \tilde S(X,\, TX \xrightarrow x X) \qquad = \qquad (SX,\, TSX
    \xrightarrow{\delta_X} STX \xrightarrow{Sx} SX)\rlap{ .}
  \end{equation*}
  and remaining data inherited
  from $\mnd S$: thus $\tilde S(f) = Sf$, ${\tilde \us}_X = \us_X$ and ${\tilde
    \ms}_X = \ms_X$. 
  %
  %
  Conversely, for a lifting of $\mnd{S}$ to $\mnd{T}$-algebras
  with action $\tilde S(X,x) = (SX, \sigma_{X,x})$, the
  corresponding distributive law $\delta \colon TS \Rightarrow ST$ is
  given by:
  \begin{equation}\label{eq:13}
    \delta_{X} = TSX \xrightarrow{TS\ut_{X}} TSTX
    \xrightarrow{\sigma_{F^\mnd{T}\!X}} STX\rlap{ .} 
  \end{equation}

  Next, 
  for a distributive law $\delta \colon TS \Rightarrow ST$, the
  corresponding extension of $\mnd{T}$ to $\C_\mnd{S}$ is given on
  objects by
  $\tilde TX = TX$ and on a Kleisli map from $X$ to $Y$~by
  \begin{equation*}
    \tilde T(X \xrightarrow{f} SY) = TX \xrightarrow{Tf} TSY
    \xrightarrow{\delta_Y} STY\rlap{, }
  \end{equation*}
  while the unit and multiplication have components
  \begin{equation*}
    X \xrightarrow{\ut_X} TX \xrightarrow{\us_{TX}} STX \qquad
    \text{and} \qquad
    TTX \xrightarrow{\mt_X} TX \xrightarrow{\us_{TX}} STX\rlap{ .}
  \end{equation*}
  Conversely, given an extension $\tilde{\mnd{T}}$ of $\mnd{T}$, we
  may view each map $1_{SX} \colon SX \rightarrow SX$ as a Kleisli map
  from $SX$ to $X$, and applying $\tilde T$ yields a Kleisli map from
  $TSX$ to $TX$, which provides the $X$-component of the corresponding
  distributive law:
  \begin{equation*}
    \tilde T(SX \xrightarrow{1_{SX}} SX) = TSX \xrightarrow{\delta_X} STX\rlap{ .}\qedhere
  \end{equation*}
\end{proof}
We can describe the algebras for the lifted monad $\tilde{\mnd S}$
associated to a distributive law in various other ways. One is in
terms of the \emph{composite monad} $\mnd{ST}$ on $\C$,
which is the monad induced by the composite adjunction
$\smash{(\C^\mnd{T})^{\tilde{\mnd S}} \leftrightarrows \C^\mnd{T}
  \leftrightarrows \C}$; its underlying endofunctor is $ST$, its unit
is $\us \ut \colon 1 \Rightarrow ST$ and its multiplication is
$\ms \mt \circ S \delta T \colon STST \Rightarrow ST$. Another is in
terms of ``$\delta$-algebras'':
\begin{Defn}
  \label{def:6}
  Let $\delta \colon TS \Rightarrow ST$ be a distributive law of
  $\mnd{S}$ over $\mnd{T}$. A \emph{$\delta$-algebra} is an
  object $X \in \C$ endowed with $\mnd{T}$-algebra structure
  $t \colon TX \rightarrow X$ and $\mnd{S}$-algebra structure
  $s \colon SX \rightarrow X$ and rendering commutative the diagram
  below. The category $\C^\delta$ of $\delta$-algebras is the full subcategory
of $\C^{\mnd S} \times_\C \C^{\mnd T}$ on the $\delta$-algebras.
\begin{equation}\label{eq:21}
  \cd[@-0.5em]{
    TSX \ar[r]^-{\delta_X} \ar[d]^-{Ts}& STX \ar[r]^-{St} &
    SX \ar[d]^-{s} \\
    TX \ar[rr]^-{t} & & X
  }
\end{equation}
\end{Defn}
The basic result relating these notions is the following; for the
proof, see~\cite{Beck1969Distributive}.

\begin{Lemma}
  \label{lem:6}
  For any distributive law $\delta \colon TS \Rightarrow ST$ of
  $\mnd{S}$ over $\mnd{T}$, there are canonical isomorphisms between 
the category of $\tilde {\mnd S}$-algebras in $\C^\mnd{T}$,
the category of $\mnd{ST}$-algebras in $\C$, and
the category of $\delta$-algebras in $\C$.
\end{Lemma}


\subsection{Weak distributive laws}
\label{sec:weak-distr-laws}
As explained in the introduction, weak distributive laws generalise
distributive laws by relaxing the axioms relating to the monad units.
There are various ways of doing this, studied in
Street~\cite{Street2009Weak} and B\"ohm~\cite{Bohm2010The-weak}, but
we will need only one, which we henceforth refer to with the unadorned name
``weak distributive law''. In the terminology
of~\cite{Bohm2010The-weak}, our notion is that of a monad in
$\cat{E\M }^w(\cat{Cat})$ whose $2$-cell data satisfy the conditions
of Lemma~1.2(3) of \emph{ibid}.
\begin{Defn}
  \label{def:4}
  Let $\mnd{S} = (S, \us, \ms)$ and $\mnd{T} = (T, \ut, \mt)$ be
  monads on a category $\C$. A \emph{weak distributive law of
    $\mnd{S}$ over $\mnd{T}$} is a natural transformation
  $\delta \colon TS \Rightarrow ST$ rendering commutative the three
  diagrams:
  \begin{equation*}
    \cd[@-0.5em@C-0.5em]{
      TSS \ar[r]^-{\delta S} \ar[d]_-{T\ms} &
      STS \ar[r]^-{S\delta} & SST \ar[d]^-{\ms T} & 
      TTS \ar[r]^-{T\delta} \ar[d]_-{\mt S} & 
      TST \ar[r]^-{\delta T} & STT \ar[d]^-{S\mt} & 
      & {T} \ar[dl]_-{T\us} \ar[dr]^-{\us T}\\
      TS \ar[rr]^-{\delta} & & ST & 
      {TS} \ar[rr]^-{\delta} & & {ST} & 
      {TS} \ar[rr]^-{\delta} & & 
      {ST}\rlap{ .}
    }
  \end{equation*}
\end{Defn}
Thus, a weak distributive law in our sense simply drops from Beck's
definition the axiom relating to the unit of $\mnd{T}$. Such weak
distributive laws correspond to \emph{weak} liftings and to
\emph{weak} extensions, where the definitions of these are a bit
more subtle. 

\begin{Defn}
  \label{def:9}
  Let $\mnd{S} = (S, \us, \ms)$ and $\mnd{T} = (T, \ut, \mt)$ be
  monads on a category $\C$. A \emph{weak lifting} of $\mnd{S}$ to
  $\C^\mnd{T}$ comprises a monad $\tilde{\mnd S}$ on $\C^\mnd{T}$ and
  natural transformations
  \begin{equation}\label{eq:10}
    U^\mnd{T} \tilde S \xRightarrow{\ \ \iota\ \ } 
    SU^\mnd{T} \xRightarrow{\ \ \pi\ \ } U^\mnd{T} \tilde S
  \end{equation}
  such that $\pi\iota = 1$, and such that each of the following diagrams
  commutes:
  \begin{equation}
    \cd[@-0.7em]{
      U^{\mnd T}\tilde S\tilde S \ar[r]^-{\iota \tilde S}
      \ar[d]_-{U^{\mnd T}\tilde \ms} &
      SU^{\mnd T}\tilde S \ar[r]^-{S\iota} &
      SSU^{\mnd T} \ar[d]^-{\ms U^{\mnd T}} & &
      & {U^\mnd{T}} \ar[dl]_-{U^\mnd{T} \tilde \us}
      \ar[dr]^-{\us U^{\mnd T}} \\
      U^{\mnd T}\tilde S \ar[rr]^-{\iota} & & SU^{\mnd T} & &
      {U^{\mnd T}\tilde S} \ar[rr]^-{\iota} & &
      {SU^{\mnd T}}
    }\label{eq:12}
  \end{equation}
  \begin{equation}
  \cd[@-0.7em]{
      SSU^{\mnd T} \ar[d]_-{\ms U^{\mnd T}} \ar[r]^-{S\pi}&
      SU^{\mnd T}\tilde S \ar[r]^-{\pi \tilde S} &
      U^{\mnd T}\tilde S\tilde S \ar[d]^-{U^\mnd{T}\tilde \ms} & &
      & {U^\mnd{T}} \ar[dr]^-{U^\mnd{T} \tilde \us}
      \ar[dl]_-{\us U^{\mnd T}} \\
      SU^{\mnd T} \ar[rr]^-{\pi} & & U^{\mnd T}\tilde S & &
      {SU^{\mnd T}} \ar[rr]^-{\pi} & &
      {U^{\mnd T}\tilde S}\rlap{ ;}
    }\label{eq:11}
  \end{equation}
  while a \emph{weak extension} of $\mnd{T}$ to
  $\C_\mnd{S}$ comprises a functor
  $\tilde T \colon \C_\mnd{S} \rightarrow \C_\mnd{S}$ and natural
  transformation $\tilde \mt \colon \tilde T\tilde T \Rightarrow \tilde T$
  such that $\tilde T \circ F_{\mnd S} = F_{\mnd S} \circ T$ and
  $\tilde \mt \circ F_{\mnd S} = F_{\mnd S} \circ \mt$.
\end{Defn}

Note that our ``weak liftings'' are exactly the simultaneous weak
$\iota$- and $\pi$-liftings of~\cite{Bohm2010The-weak}. By exactly the
same constructions as in Proposition~\ref{prop:2}, we have:
\begin{Prop}\label{prop:5}
  For monads $\mnd{S}$, $\mnd{T}$ on $\C$, there is a bijective
  correspondence between weak distributive laws of $\mnd{S}$ over
  $\mnd{T}$ and weak extensions of $\mnd T$ to $\C_{\mnd S}$.
\end{Prop}
The correspondence between weak distributive laws and weak liftings is
more interesting. It is proved by Proposition~4.4 and Theorem~4.5
of~\cite{Bohm2010The-weak} in a more general context; however, for the
particular kind of weakness we are interested in, the following more direct
proof is possible.

%

To begin with, we define a \emph{semialgebra} for a monad
$\mnd T = (T, \ut, \mt)$ to be a pair
$(X \in \C, x \colon TX \rightarrow X)$ satisfying the associativity
axiom $x.Tx = x.\mt_X$ but not necessarily the unit axiom
$x.\ut_X = 1_X$. The $\mnd{T}$-semialgebras form a category
$\C^\mnd{T}_s$, wherein a map from $(X,x)$ to $(Y,y)$ is a map
$f \colon X \rightarrow Y$ with $y.Tf = f.x$.

\begin{Lemma}
  \label{lem:4}
  If idempotents split in $\C$, then the full inclusion
  $I \colon \C^\mnd{T} \rightarrow \C^\mnd{T}_s$ has a simultaneous left
  and right adjoint
  $K \colon \C^\mnd{T}_s \rightarrow \C^\mnd{T}$.
\end{Lemma}
\begin{proof}
  For any $(X,x) \in \C^\mnd{T}_s$ we have
  $ x . \ut_X . x = x . Tx . \ut_{TX} = x .
  \mt_X . \ut_{TX} = x = x . \mt_X .
  T\ut_X = x . Tx . T\ut_X$ so that
  $x . \ut_X \colon (X,x) \rightarrow (X,x)$ is an
  idempotent of $\mnd{T}$-semialgebras. Splitting this idempotent
  yields a diagram
  \begin{equation}\label{eq:16}
    \cd{
      (X,x) \ar@{->>}[r]^-{p} & (Y,y)\ar@{ >->}[r]^-{i} & (X,x)
    }
  \end{equation}
  in $\C^\mnd{T}_s$ with $pi = 1_Y$ and
  $ip = x . \ut_X$. The semialgebra $(Y,y)$ is in fact a
  $\mnd{T}$-algebra since
  $y \ut_Y = piy\ut_Y = px . Ti .
  \ut_Y = px.\ut_X i = pipi = 1_Y$. Moreover,
  if $(Z,z)$ is a $\mnd{T}$-algebra and
  $f \colon (X,x) \rightarrow (Z,z)$, then
  $f = z\ut_Z. f = z.Tf.\ut_X
  =f.x.\ut_X = fip$ so that $f$ factors through $p$.
  On the other hand, if
  $g \colon (Z,z) \rightarrow (X,x)$, then
  $g = gz\ut_Z = x.Tg.\ut_Z =
  x.\ut_X.g = ipg$ so that $g$ factors through $i$. Thus
  $i$ and $p$ exhibit $(Y,y)$ as the value at $(X,x)$ of the desired
  left and right adjoint $K$.
\end{proof}

\begin{Prop}\label{prop:10}
  If idempotents split in $\C$, then for any monads $\mnd{S}$, $\mnd{T}$
  on $\C$, there is a bijective correspondence between weak
  distributive laws of $\mnd{S}$ over $\mnd{T}$ and weak liftings of
  $\mnd S$ to $\C^{\mnd T}$.
\end{Prop}
\begin{proof}
  Given a weak distributive law $\delta \colon TS \Rightarrow ST$, we
  may define a \emph{strict} lifting $\check{\mnd S}$ of
  $\mnd{S}$ to $\mnd{T}$-\emph{semi}algebras by taking, as in
  Proposition~\ref{prop:2}, $\check S(X,x) = (SX, Sx.\delta_X)$
  and with the remaining data inherited from $\mnd S$.
  We now obtain the desired \emph{weak} lifting $\tilde{\mnd S}$ of
  $\mnd{S}$ to $\C^\mnd{T}$ as the monad generated by the
  composite adjunction:
  \begin{equation*}
    \cd{
      {(\C^\mnd{T}_s)^{\check{\mnd S}}}
      \ar@<-4pt>[r]_-{U^{\check{\mnd S}}}
      \ar@{<-}@<4pt>[r]^-{F^{\check{\mnd S}}} \ar@{}[r]|-{\bot} &
      {{\C^\mnd{T}_s} } \ar@<-4pt>[r]_-{I} \ar@{<-}@<4pt>[r]^-{K} \ar@{}[r]|-{\bot} &
      {\C^\mnd{T}}\rlap{ .}
    }
  \end{equation*}
  In particular, $\tilde S$ sends a $\mnd{T}$-algebra $(X,x)$ to
  the $\mnd{T}$-algebra obtained as the splitting
  \begin{equation}\label{eq:14}
    \cd[@C+1em]{
      (SX, Sx.\delta_X) \ar@{->>}[r]^-{\pi_{X,x}} &
      \tilde S(X,x) \ar@{ >->}[r]^-{\iota_{X,x}} &
      (SX, Sx.\delta_X)}
  \end{equation}
  of the idempotent
  $Sx.\delta_X.\ut_{SX} \colon (SX, Sx.\delta_X) \rightarrow (SX,
  Sx.\delta_X)$ in the category of $\mnd{T}$-semialgebras. Applying
  the forgetful functor $\C^\mnd{T}_s \rightarrow \C$ to~\eqref{eq:14}
  yields the components of the $\iota$
  and $\pi$ required in~\eqref{eq:10}, and it is  clear from the manner of
  definition that the lifted unit $\tilde \us$ is the \emph{unique}
  map rendering the triangles in~\eqref{eq:12} and~\eqref{eq:11}
  commutative. As for the lifted multiplication $\tilde \ms$, a short
  calculation shows that, 
  for any $\mnd{T}$-semialgebra $(X,x)$ with
  $\mnd{T}$-algebra splitting~\eqref{eq:16}, the maps
  \begin{equation*}
    \cd[@C-0.5em]{
      (SX,Sx.\delta_X) \ar@{->>}[r]^-{Sp} &
      (SY, Sy.\delta_Y) \ar@{->>}[r]^-{\pi_{Y,y}} &
      \tilde S(Y,y) \ar@{ >->}[r]^-{\iota_{Y,y}} &
      (SY, Sy.\delta_Y) \ar@{ >->}[r]^-{Si} &
      (SX, Sx.\delta_X)
    }
  \end{equation*}
  compose to the idempotent $Sx.\delta_X.\ut_{SX}$, and so
  exhibit $\tilde S(Y,y)$ as the $\mnd{T}$-algebra splitting of
  $\check{S}(X,x) = (SX, Sx.\delta_X)$. Thus, 
  for any
  $\mnd{T}$-algebra $(X,x)$, the maps
  \begin{equation*}
    \cd[@C+0.1em]{
       \check S \check S(X,x) \ar@{->>}[r]^-{\check S\pi_{X,x}} &
       \check S \tilde S(X,x) \ar@{->>}[r]^-{\pi_{\tilde S(X,x)}} &
      \tilde S\tilde S(X,x) \ar@{ >->}[r]^-{\iota_{\tilde S(X,x)}} &
       \check S\tilde S(X,x) \ar@{->>}[r]^-{\check S\iota_{X,x}} &
       \check S \check S(X,x)
    }
  \end{equation*}
  exhibit $\tilde S \tilde S(X,x)$ as the $\mnd{T}$-algebra
  splitting of $\check {S} \check {S}(X,x)$; whence 
  $\tilde \ms$ is the \emph{unique} map rendering
  commutative the rectangles in~\eqref{eq:12} and~\eqref{eq:11}, as required.

  This concludes the construction of a weak lifting from a weak
  distributive law. Suppose conversely we have a weak
  lifting of $\mnd{S}$ to $\mnd{T}$-algebras. For each
  $\mnd{T}$-algebra $(X,x)$ with $\tilde S(X,x) = (Y,y)$, 
  define the map $\sigma_{X,x} \colon TSX \rightarrow SX$ as the
  composite
  \begin{equation*}
    TSX \xrightarrow{T\pi_{X,x}} TY \xrightarrow{y} Y
    \xrightarrow{\iota_{X,x}} SX
  \end{equation*}
  and now define $\delta \colon TS \Rightarrow ST$ to have
  components~\eqref{eq:13}. Direct calculation shows this to be a weak
  distributive law.
\end{proof}

Just as before, there are various ways of describing the algebras for
the weakly lifted monad $\tilde{\mnd S}$ associated to a weak
distributive law. We can consider the composite monad
$\widetilde{\mnd{ST}}$ induced by the adjunction
$\smash{(\C^\mnd{T})^{\tilde{\mnd S}} \leftrightarrows \C^\mnd{T}
  \leftrightarrows \C}$, whose underlying endofunctor $\widetilde{ST}$
is obtained by splitting the idempotent
\begin{equation*}
  ST \xrightarrow{\ut ST} TST \xrightarrow{\delta T} STT
  \xrightarrow{S\mt} ST\rlap{ ,}
\end{equation*}
or we can consider the category of $\delta$-algebras
defined exactly as in Definition~\ref{def:6}. The relation between
these notions is the same as before, and we record it as follows; for
the proof, we refer the reader
to~\cite[Proposition~3.7]{Bohm2010The-weak}.
\begin{Lemma}
  \label{lem:7}
  For any weak distributive law $\delta \colon TS \Rightarrow ST$ of
  $\mnd{S}$ over $\mnd{T}$, there are canonical isomorphisms between 
the category of $\tilde {\mnd S}$-algebras in $\C^\mnd{T}$,
the category of $\widetilde{\mnd{ST}}$-algebras in $\C$, and
the category of $\delta$-algebras in $\C$.
\end{Lemma}

\section{Weakly lifting the power-set monad}
\label{sec:weakly-lifting-power}
If, in the results of the previous section, we take $\C$ to be
$\cat{Set}$, $\mnd{S}$ to be the power-set monad, and $\mnd{T}$ to be
any $\cat{Set}$-monad, then we establish a bijection between (weak)
liftings of the power-set monad to $\mnd{T}$-algebras and (weak)
extensions of $\mnd{T}$ to $\cat{Set}_\ps$. Now $\cat{Set}_\ps$ is the
category $\cat{Rel}$ of \emph{sets and relations}, and the possibility
of extending structure from $\cat{Set}$ to $\cat{Rel}$ was analysed
by~\cite{Barr1970Relational}, as we now recall.

\subsection{Extending structure from sets to relations}
\label{sec:extend-struct-from}
The category $\cat{Rel}$ has sets as objects, and as morphisms
$R \colon X \tor Y$, relations $R \subseteq X \times Y$; we write
$x \mathrel R y$ to indicate that $(x,y) \in R$. Identity maps are
equality relations, and the composite of $R \colon X \tor Y$ and
$S \colon Y \tor Z$ is given by:
\begin{equation*}
  x \mathrel{SR} z \qquad \iff \qquad (\exists y \in Y)\, (x
    \mathrel R y)
  \,\wedge\, (y \mathrel S z)\rlap{ .}
\end{equation*}
Under the identification of $\cat{Rel}$ as $\cat{Set}_\ps$, the free
functor $F_\ps \colon \cat{Set} \rightarrow \cat{Set}_\ps$ corresponds
to the identity-on-objects embedding
$(\thg)_\ast \colon \cat{Set} \rightarrow \cat{Rel}$ which sends a
function $f \colon X \rightarrow Y$ to its graph
$f_\ast = \{(x, fx) : x \in X\} \subseteq X \times Y$. We also have
the reverse relation
$f^\ast = \{(fx,x) : x \in X\} \subseteq Y \times X$, and in fact,
relations of these two forms generate $\cat{Rel}$ under composition, since
every $R \colon X \tor Y$ in $\cat{Rel}$ can be written as
$q_\ast \circ p^\ast$ where
$p\colon X \leftarrow R \rightarrow Y \colon q$ are the two
projections.

Importantly, $\cat{Rel}$ is not just a category; each hom-set is
partially ordered by inclusion, and composition preserves the order on
each side, so making $\cat{Rel}$ a \emph{locally partially ordered
  $2$-category}. With respect to this structure, it is easy to see for
any function $f \colon X \rightarrow Y$ that $f_\ast$ is \emph{left
  adjoint} to $f^\ast$ in $\cat{Rel}$. This observation is key to
proving the following result, essentially due to
Barr~\cite{Barr1970Relational}; for a detailed proof,
see~\cite{Kurz2016Relation}.

\begin{Prop}
  \label{prop:4}
  Any $F \colon \cat{Set} \rightarrow \cat{Set}$ has at most one
  extension to a $2$-functor
  $\tilde F \colon \cat{Rel} \rightarrow \cat{Rel}$. This exists just
  when $F$ is weakly cartesian, and is then defined 
  on a relation $R \colon X \tor Y$ with projections
  $p\colon X \leftarrow R \rightarrow Y \colon q$ by
  \begin{equation}\label{eq:4}
    \tilde F(R) = (Fq)_\ast (Fp)^\ast \colon FX \tor FY\rlap{ .}
  \end{equation}
  Any
  $\alpha \colon F \Rightarrow G \colon \cat{Set} \rightarrow
  \cat{Set}$ has at most one extension to a $2$-natural 
  $\tilde \alpha \colon \tilde F \Rightarrow \tilde G$. This exists just
  when $\alpha$ is weakly cartesian, and has 
  components $(\tilde \alpha)_X = (\alpha_X)_\ast$.
\end{Prop}

Here, a functor $F \colon \cat{Set} \rightarrow \cat{Set}$ is
\emph{weakly cartesian} if it preserves weak pullback squares, and a
natural transformation $\alpha \colon F \Rightarrow G$ is \emph{weakly
  cartesian} if its naturality squares are weak pullbacks; recall that
a weak pullback square is one for which
the induced comparison map into the pullback
is an epimorphism.
  
%
\begin{Cor}
  \label{cor:1}
  For any monad $\mnd{T} = (T, \ut, \mt)$ on $\cat{Set}$:
  \begin{enumerate}[(i)]
  \item If $T$, $\ut$ and $\mt$ are all weakly cartesian, then there
    is a canonical extension of $\mnd{T}$ to $\cat{Rel}$, and so by
    Proposition~\ref{prop:2}, a canonical lifting of $\ps$ to
    $\mnd{T}$-algebras;
  \item If only $T$ and $\mt$ are weakly cartesian, then there is
    still a canonical \emph{weak} extension of $\mnd{T}$ to
    $\cat{Rel}$, and so a canonical \emph{weak} lifting of $\ps$ to
    $\mnd{T}$-algebras.
  \end{enumerate}
\end{Cor}
The intended application of this takes $\mnd{T}$ to be the ultrafilter
monad, but before turning to this, we consider two simpler
examples. 

\subsection{First example}
\label{sec:first-example}
Let $\mnd{T} = (T, \ut, \mt)$ be the commutative monoid monad. This is
an \emph{analytic monad} in the sense
of~\cite{Joyal1986Foncteurs}---in fact, the terminal one---so that
each of $T$, $\ut$ and $\mt$ is weakly cartesian: thus $\mnd{T}$
extends \emph{strictly} from $\cat{Set}$ to $\cat{Rel}$. Using the
formula~\eqref{eq:4}, we see that the action of this extension on a
relation $R \colon X \tor Y$ is the relation
$\tilde TR \colon TX \tor TY$ with
\begin{equation*}
  x_1 \cdots x_n \mathrel{\tilde TR} y_1 \cdots y_m \iff
  (\exists\, \sigma \colon\! n \cong m)(x_1 \mathrel R y_{\sigma(1)})
  \wedge \dots \wedge (x_n \mathrel R y_{\sigma(n)})\rlap{ .}
\end{equation*}

Plugging this in to the proof of Proposition~\ref{prop:2}, we see that
the distributive law corresponding to this extension has components
$\delta_X \colon TPX \rightarrow PTX$ given by
\begin{equation}\label{eq:8}
  A_1 \cdots A_n \quad \mapsto \quad \{a_1 \cdots a_n : \text{each }a_i \in A_i\}\rlap{ .}
\end{equation}
and so that, under the identification of $\mnd{T}$-algebras with
commutative monoids, the lifted monad $\tilde{\ps}$ takes a
commutative monoid $(X, \cdot, 1)$ to the commutative monoid with
underlying set $PX$, unit $\{1\}$ and multiplication
$A \cdot B = \{a \cdot b : a \in A, b \in B\}$. The algebras for the
lifted monad $\tilde{\ps}$ are precisely the \emph{commutative
  unital quantales}: complete lattices $X$ endowed with a commutative
monoid structure $(X, \cdot, 1)$ whose binary multiplication preserves
sups separately in each variable.

\subsection{Second example}
\label{sec:second-example-1}
We now consider the finite power-set monad $\mnd{P}_f$ on
$\cat{Set}$, whose algebras are idempotent commutative monoids. Unlike
the commutative monoid monad, this does not extend strictly from
$\cat{Set}$ to $\cat{Rel}$, due to:
\begin{Lemma}
  \label{lem:5}
  The endofunctor and the multiplication of the finite power-set monad
  $\ps_f$ on $\cat{Set}$ are weakly cartesian, but the unit is not.
\end{Lemma}
(In fact the same is true on replacing $\ps_f$ by the full power-set
monad $\ps$. In other words, $\ps$ does not distribute over itself;
see~\cite{Klin2018Iterated}.)
\begin{proof}
  For the endofunctor part, see, for
  instance,~\cite[Proposition~1.4]{Johnstone2001On-the-structure}. For
  the multiplication, we must show that, given a function
  $f \colon X \rightarrow Y$, a finite subset $A \subseteq X$ and a
  finite subset $B = \{B_1, \dots, B_n\} \subseteq P_f Y$ with
  $f(A) = B_1 \cup \dots \cup B_n$, there exists a finite subset
  $\{C_1, \dots, C_m\} \subseteq P_f X$ with
  $\{f(C_1), \cdots, f(C_m)\} = B$ and $A = C_1 \cup \dots \cup C_m$.
  We have such on taking $m = n$ and $C_i = A \cap f^{-1}(B_i)$.
  Finally, to see the unit is not weakly cartesian, note that under
  the function $\{0,1\} \rightarrow \{0\}$, the finite
  $\{0,1\} \subseteq \{0,1\}$ maps to the singleton $\{0\}$, but is
  not itself a singleton.
\end{proof}
However, we still have a weak extension of $\ps_f$ to $\cat{Rel}$;
this observation is apparently due to Ehrhard, and is discussed in
detail in~\cite{Hyland2006A-Category}. Calculating explicitly
using~\eqref{eq:4}, we see that the action of $\tilde{P}_f$ on a
relation $R \colon X \tor Y$ is given by the ``Egli--Milner relation'':
\begin{equation*}
  A \mathrel{\tilde P_f R} B \iff (\forall a \in A.\, \exists b \in
  B.\, a \mathrel R b) \wedge (\forall b \in B.\, \exists a \in
  A.\, a \mathrel R b)\rlap{ .}
\end{equation*}
Thus, by Proposition~\ref{prop:5}, the weak distributive law
corresponding to this weak extension has components
$\delta_X \colon P_f PX \rightarrow P P_f X$ given by
\begin{equation*}
  \A \  \mapsto  \ \{ B \subseteq X \text{
    finite} : B \subseteq \textstyle\bigcup \A \text{ and } A \cap B \neq
  \emptyset \text{ for all $A \in \A$} \}\rlap{ .}
\end{equation*}

We now calculate the corresponding weak lifting of the power-set monad
to the category of $\ps_f$-algebras. Given such an algebra $(X,x)$, we
first form the associated semialgebra $(PX, Px.\delta_X)$, whose
action map $P_f PX \rightarrow PX$ is
\begin{align*}
  \{A_1, \dots, A_n\} & \mapsto \{a_1 \cdots a_m : \text{each $a_i$ is
  in some $A_j$, and some $a_i$ is in each $A_j$}\}\rlap{ .}
\end{align*}
In particular, the idempotent
$Px.\delta_X.\ut_X \colon PX \rightarrow PX$ takes
$A \subseteq X$ to the set of non-empty finite products of elements of
$A$. Clearly $A$ is fixed by this idempotent just when it is a
\emph{subsemigroup}---i.e., closed under binary multiplication.

It follows that the lifted monad $\tilde{\ps}$ acts on $(X,x)$
to yield the set $P_\bullet(X)$ of all subsemigroups of $X$, under the
$P_f$-algebra structure given as in the previous display. Reading off
the monoid structure from this, we see that the unit of $P_\bullet X$
is given by $\{1\}$, while the binary multiplication is given by
\begin{equation*}
  A \cdot B = \{a_1 \cdots a_n \cdot b_1 \cdots b_m : n,m \geqslant 1
  \text{, each $a_i \in A$ and each $b_j \in B$}\}\rlap{ .}
\end{equation*}
In this expression, since $A$ and $B$ are already subsemigroups, we
have that $a = a_1 \cdots a_n$ is itself in $A$ and
$b = b_1 \cdots b_m$ is itself in $B$; so, more succinctly,
\begin{equation}\label{eq:18}
  A \cdot B = \{a \cdot b : a \in A, b \in B\}\rlap{ ,}
\end{equation}
i.e., the same formula that we derived in
Section~\ref{sec:first-example} for the commutative
monoid monad.  
It now follows from Lemma~\ref{lem:7} that the algebras for the lifted
monad $\tilde{\ps}$ are exactly the commutative (unital) quantales
whose multiplication is idempotent.

This example can be extended in various directions. Firstly, recall
that a \emph{normal band} is an idempotent semigroup satisfying the
axiom $xyzw = xzyw$. The free normal band on a set $X$ is the set
$P_f^{\ast \ast} X$ of \emph{bipointed} finite subsets of $X$ under
the multiplication $(A,a,b) \cdot (B,c,d) = (A \cup B,a,d)$. The
induced monad $\ps_f^{\ast\ast}$ does not have weakly cartesian unit,
but has weakly cartesian endofunctor and multiplication; so we have a
weak lifting of $\ps$ to the category of normal bands. Like before,
$\smash{\tilde{\ps}}$ takes a normal band $X$ to the normal band
$P_\bullet X$ of sub-semigroups under the binary
operation~\eqref{eq:18}. This construction is also given
in~\cite{Zhao2002Idempotent}, but without the monadic context.

A second direction of extension is to consider the monad $\mnd{T}_S$
on $\cat{Set}$ whose algebras are semimodules over a given commutative
semiring $S$. In~\cite[Theorem~8.10]{Clementino2014The-monads}
conditions are given on $S$ which characterise \emph{precisely} when
the associated $\mnd{T}_S$ has weakly cartesian endofunctor and
multiplication; under these conditions, then, we obtain a weak lifting
of the power-set monad to the category of $S$-modules. Our two
preceding examples fit into this framework on taking
$S = (\mathbb{N}, \times, +)$ and $S = (\{0,1\}, \wedge, \vee)$; as
was shown in~\cite[Example~9.5]{Clementino2014The-monads}, other
legitimate choices of $S$ include $(\mathbb{Q}_+, \times, +)$ and
$(\mathbb{R}_+, \times, +)$.

%
%
%
%

\section{The Vietoris monad and weak distributive laws}
\label{sec:vietoris-monad-via}

\subsection{Recovering the Vietoris monad}
\label{sec:recov-viet-monad}
We now prove our main theorem, recovering the Vietoris monad as the
weak lifting of the power-set monad associated to the canonical weak
distributive law of $\ps$ over $\uf$. We begin with the following
well known result; see, for example,~\cite[Examples~III.1.12.3 and
Proposition~III.1.12.4]{Hofmann2014Lax-algebras}.
\begin{Lemma}
  \label{lem:1}
  The endofunctor and multiplication of the monad $\uf$ are
  weakly cartesian, but the unit is not.
\end{Lemma}
As such, we have a canonical weak extension of $\uf$ to $\cat{Rel}$.
The action of $\tilde \beta \colon \cat{Rel} \rightarrow \cat{Rel}$ on
a relation $R \colon X \tor Y$ is given by
\begin{equation}\label{eq:3}
  \F \mathrel{\tilde \beta R} \G \iff (A \in \F \implies R(A) \in \G)
\end{equation}
where we write $R(A)$ for
$\{y \in Y : (\exists a \in A) (a \mathrel R y)\}$;
see, for example, \cite[Examples
III.1.10.3(3)]{Hofmann2014Lax-algebras}. Corresponding to this weak
extension, we have a weak distributive law of $\ps$ over $\uf$;
calculating from the above expression, we see that its components
$\delta_X \colon \beta PX  \rightarrow P \beta X$ are
given by
\begin{equation}\label{eq:7}
    \delta_X(\mathbf F) = \{\,\F \in \beta X : \textstyle\bigcup \A \in
    \F \text{ for all }\A \in \mathbf F\,\}\rlap{ .}
\end{equation}

We now wish to calculate the associated weak lifting of $\ps$ to
$\uf$-algebras, i.e., to compact Hausdorff spaces. We begin with:

\begin{Lemma}
  \label{lem:3}
  Let $(X, \xi \colon \beta X \rightarrow X)$ be a $\uf$-algebra. The
  action map $\beta PX \rightarrow PX$ of the semialgebra
  $(PX, P\xi.\delta_X)$ is given by
  \begin{equation}\label{eq:6}
    \mathbf F \ \mapsto\  \textstyle\bigcap_{\A \in \mathbf F} \overline{\textstyle \bigcup
      \A}
  \end{equation}
  where $(\overline{\phantom{x}})$ denotes closure in the topology on
  $X$. The idempotent function
  $P\xi.\delta_X.\eta_{PX} \colon (PX, P\xi.\delta_X) \rightarrow (PX,
  P\xi.\delta_X)$ sends each $B \in PX$ to its closure.
\end{Lemma}
\begin{proof}
  Given $x \in X$, we have
  $x \in \bigcap_{\A \in \mathbf F} \overline{\bigcup \A}$ if and only if each
  open neighbourhood of $x$ meets each $\bigcup \A$, if and only if
  there exists an ultrafilter containing each $\bigcup \A$ and
  converging to $x$. But by~\eqref{eq:7}, this happens just when
  $x \in P\xi(\delta_X(\mathbf F))$. Finally, since
  $\eta_{PX} \colon PX \rightarrow \beta PX$ sends $B$ to the
  ultrafilter $\{ \A \subseteq PX : B \in \A\}$, the idempotent
  $P\xi.\delta_X.\eta_{PX}$ sends each $B \in PX$ to
  $\bigcap_{B \in \A}\overline{\bigcup \A} = \overline B$.
\end{proof}

\begin{Thm}
  \label{thm:1}
  The Vietoris monad on the category of compact Hausdorff spaces is
  the weak lifting of the power-set monad associated to the canonical
  weak distributive law of the power-set monad over the ultrafilter
  monad.
\end{Thm}

\begin{proof}
  Let $X$ be a compact Hausdorff space seen as a $\uf$-algebra
  $(X, \xi \colon \beta X \rightarrow X)$. The $\uf$-algebra
  $\tilde P(X, \xi)$ is obtained by splitting the idempotent
  $P\xi.\delta_X. \eta_{PX}$ on $(PX, P\xi.\delta_X)$; so by the
  previous lemma, its underlying set is the set $VX$ of \emph{closed}
  subsets of $X$, and its $\uf$-algebra structure is given by the same
  formula as in~\eqref{eq:6}. By naturality in~\eqref{eq:10}, the
  action of $\tilde P$ on maps is given by direct image, while by the
  formulae in~\eqref{eq:12}, the unit and multiplication of
  $\tilde{\ps}$ are once again given by inclusion of singletons and
  set-theoretic union.
  
  As such, it remains only to show that the $\uf$-algebra structure on
  $\tilde P(X,\xi)$ describes the Vietoris topology; in light of
  Lemma~\ref{lem:3}, we must thus show that any $\F \in \beta VX$
  converges in the Vietoris topology to the unique point
  $L = \textstyle\bigcap_{A \in \F} \overline{\textstyle\bigcup A}$.
  This follows from Lemma~\ref{lem:15} below, since the
  Vietoris topology on $VX$ is the Lawson topology on the continuous
  lattice $(VX, \supseteq)$.
\end{proof}

\subsection{Vietoris algebras}
\label{sec:vietoris-algebras}
The composite monad associated to the weak distributive law of $\ps$
over $\vi$ is easily seen to be the well known \emph{filter monad}
$\flt$; as such, Lemma~\ref{lem:7} asserts a canonical isomorphism
between the categories of $\vi$-algebras in $\KH$ and of
$\flt$-algebras in $\cat{Set}$. This was originally proven
as~\cite[Theorem~6.3]{Wyler1981Continuous} and is, in fact, how Wyler
identified the $\vi$-algebras as the continuous lattices---by first
identifying the $\flt$-algebras as such (a result originally proved by
Day~\cite{Day1975Filter}).

Now Lemma~\ref{lem:7} also identifies $\vi$-algebras with
$\delta$-algebras for the weak distributive law $\delta \colon \beta
P \Rightarrow P \beta$, i.e., as sets $X$ endowed with $\uf$-algebra
structure $\xi \colon \beta X \rightarrow X$ and $\ps$-algebra
structure $i \colon PX \rightarrow X$ subject to commutativity in
\begin{equation}\label{eq:22}
  \cd[@-0.5em]{
    \beta PX \ar[r]^-{\delta_X} \ar[d]^-{\beta i}& P \beta X \ar[r]^-{P\xi} &
    PX \ar[d]^-{i} \\
    \beta X \ar[rr]^-{\xi} & & X\rlap{ .}
  }
\end{equation}

In~\cite{Wyler1981Continuous}, Wyler does note that a $\vi$-algebra is
a $\uf$-algebra and a $\ps$-algebra subject to some
compatibility---see, for example, Proposition~6.4 of
\emph{ibid.}---but does not express this in terms of the
square~\eqref{eq:22}. In fact, by using~\eqref{eq:22} it is 
easy to give a \emph{direct} proof that Vietoris algebras
are continuous lattices, as we will now do.

In what follows, given a filter $\F$ on a topological space, we write
$\adh \F$ for $\bigcap_{A \in \F} \overline A$; recall that, for an
ultrafilter $\F$, the points in $\adh \F$ are precisely those to which
$\F$ converges. On the other hand, for a filter $\F$ on a complete
lattice, we write $\liminf \F$ for $\sup \{ \inf A : A \in \F\}$.

%
%
\begin{Prop}
  \label{prop:8}
  Let $X$ be a complete lattice and a compact Hausdorff space, seen as
  a $\uf$-algebra $\xi \colon \beta X \rightarrow X$ via ultrafilter
  convergence and as a $\ps$-algebra $i \colon PX \rightarrow X$
  by taking infima. The following are equivalent:
  \begin{enumerate}[(i)]
  \item $(X, \xi, i)$ is a $\delta$-algebra, i.e.,
    renders~\eqref{eq:22} commutative.
  \item $\liminf \F = \inf \adh \F$ for any filter $\F$ on $X$.
  \end{enumerate}
\end{Prop}
\begin{proof}
  \emph{(ii) $\Rightarrow$ (i)}. Let $\mathbf{F} \in \beta \P X$. By
  Lemma~\ref{lem:3}, the upper path around~\eqref{eq:22} takes
  $\mathbf{F}$ to $\inf \adh \bigcup \mathbf{F}$ where
  $\bigcup\mathbf{F}$ is the filter generated by $\bigcup \A$ for each
  $\A \in \mathbf{F}$. The lower path takes $\mathbf{F}$ to the limit
  point $\inf \adh \F$ of the ultrafilter $\F$ generated by the sets
  $\A^i = \{\inf A : A \in \A\}$ for all $\A \in \mathbf{F}$. So given
  (ii), it suffices to show that $\liminf \bigcup \mathbf F =
  \liminf \F$, which follows since $\inf \A^i = \inf(\bigcup
  \A)$ for all $\A \in \mathbf{F}$.

  \emph{(i) $\Rightarrow$ (ii)} We first show that for all $a \in X$,
  the principal upset $\mathop\uparrow a$ and downset
  $\mathop\downarrow a$ are closed. Consider for any ultrafilter $\F$
  on $X$ the ultrafilter $\mathbf{F}$ on $\P X$ generated by the sets
  $\{a, B\} = \{\{a, b\} : b \in B\}$ for all $B \in \F$. Note that:
  \begin{itemize}
\item If $\mathop \downarrow a \in \F$, then for each
  $B \in \F$ also $B_a = B \cap {\mathop \downarrow a} \in \F$. Thus
  $B \supseteq B_a = i_!(\{a,B_a\}) \in Pi(\mathbf{F})$, so that
  $Pi(\mathbf{F}) = \F$ and $\xi(Pi(\mathbf{F})) = \xi(\F)$.
  \item If $\mathop\uparrow a \in \F$, then $\{a\} \in Pi(\mathbf F)$
    whence $\xi(Pi(\mathbf F)) = a$.
  \end{itemize}
  In either case,
  $P\xi(\delta_X(\mathbf{F})) = \bigcap_{B \in \F}\overline{\{a\} \cup
    B} = \{a\} \cup \bigcap_{B \in \F} \overline B = \{a, \xi(\F)\}$
  by Lemma~\ref{lem:3}, and so
  $i(P\xi(\delta_X(\mathbf{F}))) = a \wedge \xi(\F)$. So by the
  assumption, if $\mathop \downarrow a \in \F$ then
  $\xi(\F) = a \wedge \xi(\F)$ so that
  $\xi(\F) \in \mathop \downarrow a$; while if
  $\mathop \uparrow a \in \F$ then $a = a \wedge \xi(\F)$ so that
  $\xi(\F) \in \mathop \uparrow a$. This proves that both $\mathop \downarrow a$ and
  $\mathop \uparrow a$ are closed.


  We now prove (ii). Given a filter $\F$ on $X$, the family of subsets
  of $\P X$ given by $\F$ together with $\mathop \downarrow A$ for
  each $A \in \F$ has the finite intersection property; let
  $\mathbf{F}$ be any ultrafilter on $\P X$ which extends it. 
  Now, for each $A \in \F$ we have
  $\mathop \downarrow A \in \mathbf{F}$ and so
  $A = \bigcup \mathop \downarrow A \in \bigcup \mathbf{F}$. On the
  other hand, each $\A \in \mathbf{F}$ meets $\F$, say in $A$ and so
  $\bigcup \A \supseteq A$ is in $\F$. So $\bigcup \mathbf{F} = \F$,
  and so by Lemma~\ref{lem:3} the upper path around~\eqref{eq:22}
  takes $\mathbf{F}$ to $\inf \adh \F$. As for the lower path,
  $\beta i(\mathbf{F})$ contains $\{\inf A : A \in \F\}$ and
  $\mathop\uparrow(\inf A)$ for each $A \in \F$; so
  $\xi(\beta i(\mathbf{F}))$ is contained in the intersection of
  closed sets
  $\bigcap \mathop\uparrow(\inf A) = \mathop\uparrow(\liminf \F)$, but
  also in the closed set
  $\mathop \downarrow(\liminf \F) \supseteq \{\inf A : A \in \F\}$ and
  so must equal $\liminf \F$. Thus $\inf \adh \F = \lim \inf \F$ as
  desired.
\end{proof}

The remainder of the argument is
standard continuous lattice theory, contained in,
say,~\cite{Gierz1980A-compendium}; we include it here for the
sake of a self-contained presentation.
\begin{Lemma}
  \label{lem:15}
  An ultrafilter $\F$ on a continuous lattice converges in the Lawson
  topology to the unique point $\ell = \liminf \F$.
\end{Lemma}
\begin{proof}
  We first show $\F$ contains every subbasic open neighbourhood of
  $\ell$. First, if $\ell \in s^+$, i.e., $s \ll \ell$, then
  $s \ll \inf A$ for some $A \in \F$, and so $s \ll a$ for all
  $a \in A$; whence $A \subseteq s^+$ and so $s^+ \in \F$. Next, if
  $\ell \in s^-$, i.e., $s \nleqslant \ell$ then $s \nleqslant \inf A$
  for all $A \in \F$. So for each $A \in \F$, we have $s \nleqslant a$
  for some $a \in A$, i.e., each $A \in \F$ meets $s^-$, and so, since
  $\F$ is an ultrafilter, we have $s^- \in \F$. Thus $\F$ converges to
  $\ell$; suppose it also converges to $y$. Then for each $s \ll y$ we
  have $s^+ \in \F$ and so $s \leqslant \inf s^+ \leqslant \ell$. Since
  $y = \bigvee \{s : s \ll y\}$ we must have $y \leqslant \ell$. We
  claim $\ell \leqslant y$, i.e., $\inf A \leqslant y$ for each
  $A \in \F$. If not, then $\inf A \nleqslant y$ for some $A \in \F$,
  so that $(\inf A)^-$ is in $\F$. So $(\inf A)^-$ and $A$ are
  disjoint sets in $\F$, a contradiction.
\end{proof}

\begin{Prop}\label{prop:9}
  Let $X$ be a complete lattice and a compact Hausdorff space. The
  following are equivalent:
  \begin{enumerate}[(i)]
  \item $\liminf \F = \inf \adh \F$ for any filter $\F$ on $X$;
  \item $X$ is a continuous lattice and its topology is the Lawson topology.
  \end{enumerate}
\end{Prop}
\begin{proof}
  \emph{(i) $\Rightarrow$ (ii)}. By the proof of
  Proposition~\ref{prop:8}, each principal upset $\mathop \uparrow x$
  is closed in $X$. We now show that, if $U$ is up-closed and open and
  $x \in U$, then $\inf U \ll x$. Indeed, suppose $x \leqslant \sup D$
  for some directed $D \subseteq X$. Then
  $\mathop\uparrow(\sup D) \subseteq U$ since $U$ is up-closed, and so
  $\emptyset = (X \setminus U) \cap \mathop \uparrow(\sup D) =(X
  \setminus U) \cap \bigcap\{\mathop \uparrow d : d \in D\}$. By
  compactness of $X$ and downward-directedness of
  $\{\mathop \uparrow d:d \in D\}$, it follows that
  $\emptyset = (X \setminus U) \cap \mathop \uparrow d$ for some
  $d \in D$; thus $d \in U$ and so $\inf U \leqslant d$ as required.


  We now show that $X$ is continuous. Let $x \in X$ and let $\F$ be
  the neighbourhood filter of $x$. Since $X$ is Hausdorff, we have
  $\adh \F = \{x\}$ and so by (i) that $\liminf \F = x$. Clearly
  $\liminf \F$ is the supremum of
  $\{ \inf U : U \text{ up-closed in } \F\}$, and by above
  $\inf U \ll x$ for each such $U$. It follows that
  $x = \sup\{s : s \ll x\}$ so that $X$ is continuous. Finally, the
  condition $\liminf \F = \inf \adh \F$ applied to an ultrafilter
  implies by Lemma~\ref{lem:15} that the topology on $X$ is the Lawson
  topology.

  \emph{(ii) $\Rightarrow$ (i)}. We first show
  $\inf A = \inf \overline A$ for any $A \subseteq X$. Clearly
  $\inf \overline A \leqslant \inf A$; while if $x \in \overline A$,
  then $x$ is the convergence point of some ultrafilter $\F$ containing
  $A$, whence $x = \liminf \F \geqslant \inf A$, so that
  $\inf A \leqslant \inf \overline A$. We now prove (i). Given a
  filter $\F$, we have for each $A \in \F$ that
  $\inf A = \inf \overline A \leqslant \inf \adh \F$ and so
  $\liminf \F \leqslant \inf \adh \F$. It remains to show
  $\inf \adh \F \leqslant \liminf \F$. By continuity, we can write
  $\inf \adh \F$ as $\bigvee\{s : s \ll \inf \adh \F\}$, so it
  suffices to show $s \ll \inf \adh \F$ implies
  $s \leqslant \liminf \F$. We prove the contrapositive: if
  $s \nleqslant \liminf \F$ then $s \not\ll \inf\adh\F$. Now if
  $s \nleqslant \liminf \F$, then
  $s \nleqslant \inf A = \inf \overline A$ for each $A \in \F$. Thus,
  for each $A \in \F$ there is some $a \in \overline{A}$ with
  $s \nleqslant a$ and hence $s \not \ll a$. This says that the closed
  set $X \setminus s^+$ meets $\bar A$ for each $A \in \F$; whence
  $\{\overline A : A \in \F\} \cup \{X \setminus s^+\}$ has the finite
  intersection property, so that by compactness, $X \setminus s^+$
  meets $\bigcap_{A \in \F} \overline A = \adh \F$. This means
  $s \not \ll a$ for some $a \in \adh \F$, and so
  $s \not \ll \inf \adh \F$ as desired.
\end{proof}

%

So $\delta$-algebras are continuous lattices, and it is easy to
identify the corresponding $\delta$-algebra maps as the inf- and
directed-sup preserving functions. We thus recover:
\begin{Thm}
  \label{thm:5}\cite{Wyler1981Continuous}
  The category of $\vi$-algebras is isomorphic to the category of
  continuous lattices with inf- and directed sup-preserving maps.
\end{Thm}

\subsection{Variations}
\label{sec:variations}

It is natural to consider variations on Theorem~\ref{thm:1}, involving
different weak distributive laws on possibly different categories.
Treating these in detail must await further work, but it is
worth sketching a couple of possibilities.

On the one hand, we may replace $\ps$ by the \emph{non-empty}
power-set monad $\ps_{+}$ on $\cat{Set}$, while keeping $\uf$ the
same. In this case, we expect to obtain a weak distributive law of
$\ps_+$ over $\uf$ whose corresponding weak lifting to the category
$\KH$ of $\uf$-algebras is the \emph{proper} Vietoris monad $\vi_+$.
This monad, considered in~\cite{Wyler1985Algebraic}, sends a compact
Hausdorff space $X$ to its set of non-empty closed subsets under the
Vietoris topology, and has as its algebras the continuous
\emph{semilattices}.

On the other hand, we can replace $\ps$ by the \emph{upper-set monad}
$\ps^\uparrow$ on the category of posets, and $\uf$ by the \emph{prime
  filter monad}~$\mnd{Pf}$. As in~\cite{Flagg1997Algebraic}, this
latter monad has as algebras the \emph{compact pospaces}---compact
spaces $X$ with a partial order $\leqslant$ which is closed in
$X \times X$. In this case, via the partially ordered version of
Barr's relation lifting~\cite[Section~3.3]{Kurz2016Relation}, we
expect to find a weak distributive law of $\ps^\uparrow$ over
$\mnd{Pf}$, with corresponding weak lifting the ``ordered Vietoris
monad'' $\mnd{V}^\uparrow$ on compact pospaces. This takes a compact
pospace $X$ to its space $V^{\uparrow} X$ of closed upper-sets ordered
by reverse inclusion, with a modified version of the Vietoris
topology; see~\cite[Example VI-3.10]{Gierz1980A-compendium}. As
explained in~\cite{Hofmann2014The-enriched}, the
$\mnd{V}^\uparrow$-algebras are, once again, the continous lattices.
In fact,~\cite{Hofmann2014The-enriched} describes other Vietoris-like
monads, and it is natural to hope that these may arise in a
similar manner; but again, we leave this to future work.

\end{document}